\documentclass[10pt,twoside]{amsart}
\usepackage{graphicx}
\usepackage{amssymb}
\usepackage{amsmath}
\usepackage{amsthm}
\usepackage{amscd}
\usepackage[all]{xy}
\usepackage{dsfont}

\def\YEAR{\year}\newcount\VOL\VOL=\YEAR\advance\VOL by-1995
\def\firstpage{1}\def\lastpage{5}
\def\received{}\def\revised{}
\def\communicated{}

\makeatletter
\def\magnification{\afterassignment\m@g\count@}
\def\m@g{\mag=\count@\hsize6.5truein\vsize8.9truein\dimen\footins8truein}
\makeatother

\font\eightrm=cmr8
\font\caps=cmcsc10                    
\font\Caps=cmcsc10 scaled \magstep1   

\def\bfseries{\normalsize\caps}

\makeatletter
\@specialpagefalse\headheight=8.5pt
\def\DocMath{}
\renewcommand{\@evenhead}{%
    \ifnum\thepage>\lastpage\rlap{\thepage}\hfill%
    \else\rlap{\thepage}\slshape\leftmark\hfill{\caps\SAuthor}\hfill\fi}%
\renewcommand{\@oddhead}{%
    \ifnum\thepage=\firstpage{\DocMath\hfill\llap{\thepage}}%
    \else{\slshape\rightmark}\hfill{\caps\STitle}\hfill\llap{\thepage}\fi}%
\makeatother

\def\TSkip{\bigskip}
\newbox\TheTitle{\obeylines\gdef\GetTitle #1
\ShortTitle  #2
\SubTitle    #3
\Author      #4
\ShortAuthor #5
\EndTitle
{\setbox\TheTitle=\vbox{\baselineskip=20pt\let\par=\cr\obeylines%
\halign{\centerline{\Caps##}\cr\noalign{\medskip}\cr#1\cr}}%
    \copy\TheTitle\TSkip\TSkip%
\def\next{#2}\ifx\next\empty\gdef\STitle{#1}\else\gdef\STitle{#2}\fi%
\def\next{#3}\ifx\next\empty%
    \else\setbox\TheTitle=\vbox{\baselineskip=20pt\let\par=\cr\obeylines%
    \halign{\centerline{\caps##} #3\cr}}\copy\TheTitle\TSkip\TSkip\fi%
\centerline{\caps #4}\TSkip\TSkip%
\def\next{#5}\ifx\next\empty\gdef\SAuthor{#4}\else\gdef\SAuthor{#5}\fi%
\ifx\received\empty\relax
    \else\centerline{\eightrm Received: \received}\fi%
\ifx\revised\empty\TSkip%
    \else\centerline{\eightrm Revised: \revised}\TSkip\fi%
\ifx\communicated\empty\relax
    \else\centerline{\eightrm Communicated by \communicated}\fi\TSkip\TSkip%
\catcode'015=5}}\def\Title{\obeylines\GetTitle}
\def\Abstract{\begingroup\narrower
    \parskip=\medskipamount\parindent=0pt{\caps Abstract. }}
\def\EndAbstract{\par\endgroup\TSkip}

\long\def\MSC#1\EndMSC{\def\arg{#1}\ifx\arg\empty\relax\else
     {\par\narrower\noindent%
     2000 Mathematics Subject Classification: #1\par}\fi}

\long\def\KEY#1\EndKEY{\def\arg{#1}\ifx\arg\empty\relax\else
    {\par\narrower\noindent Keywords and Phrases: #1\par}\fi\TSkip}

\newbox\TheAdd\def\Addresses{\vfill\copy\TheAdd\vfill
    \ifodd\number\lastpage\vfill\eject\phantom{.}\vfill\eject\fi}
{\obeylines\gdef\GetAddress #1
\Address #2
\Address #3
\Address #4
\EndAddress
{\def\xs{4.3truecm}\parindent=0pt
\setbox0=\vtop{{\obeylines\hsize=\xs#1\par}}\def\next{#2}
\ifx\next\empty 
     \setbox\TheAdd=\hbox to\hsize{\hfill\copy0\hfill}
\else\setbox1=\vtop{{\obeylines\hsize=\xs#2\par}}\def\next{#3}
\ifx\next\empty 
     \setbox\TheAdd=\hbox to\hsize{\hfill\copy0\hfill\copy1\hfill}
\else\setbox2=\vtop{{\obeylines\hsize=\xs#3\par}}\def\next{#4}
\ifx\next\empty\ 
     \setbox\TheAdd=\vtop{\hbox to\hsize{\hfill\copy0\hfill\copy1\hfill}
                \vskip20pt\hbox to\hsize{\hfill\copy2\hfill}}
\else\setbox3=\vtop{{\obeylines\hsize=\xs#4\par}}
     \setbox\TheAdd=\vtop{\hbox to\hsize{\hfill\copy0\hfill\copy1\hfill}
            \vskip20pt\hbox to\hsize{\hfill\copy2\hfill\copy3\hfill}}
\fi\fi\fi\catcode'015=5}}\gdef\Address{\obeylines\GetAddress}

\newtheorem{thm}[subsection]{Theorem}
\newtheorem{cor}[subsection]{Corollary}
\newtheorem{lem}[subsection]{Lemma}

\newtheorem{fact}[subsection]{Fact}

\newtheorem{prop}[subsection]{Proposition}
\theoremstyle{definition}
\newtheorem{defn}[subsection]{Definition}
\newtheorem{clm}[subsection]{Claim}
\theoremstyle{remark}
\newtheorem{rem}[subsection]{Remark}

\numberwithin{equation}{subsection}

\makeatletter

\newcommand{\Rmnum}[1]{\expandafter\@slowromancap\romannumeral #1@}
\makeatother

\makeatletter
\renewenvironment{proof}[1][\proofname]{\par
    \pushQED{\qed}%
    \normalfont \topsep6\p@\@plus6\p@ \labelsep1em\relax
    \trivlist
    \item[\hskip\labelsep
        \bfseries #1]\ignorespaces
}{%
    \popQED\endtrivlist\@endpefalse
} \makeatother
\renewcommand{\proofname}

\begin{document}
\Title

         An improvement of de Jong--Oort's purity theorem
\ShortTitle

\SubTitle

\Author
        Yanhong Yang
\ShortAuthor
           Yanhong Yang
\EndTitle

\Abstract Consider an $F$-crystal over a noetherian scheme $S$. De
Jong--Oort's purity theorem states that the associated Newton
polygons over all points of $S$ are constant if this is true outside
a subset of codimension bigger than 1. In this paper we show an
improvement of the theorem, which says that the Newton polygons over
all points of $S$ have a common break point if this is true outside
a subset of codimension bigger than 1.

\EndAbstract

\MSC Primary 14F30; Secondary 11F80,14H30.
\EndMSC

\KEY F-crystal, Newton slope, Galois representation.
\EndKEY

\Address
    Yanhong Yang
    Department of Mathematics
    Columbia University
    New York
    NY 10027 (USA)
    yhyang@math.columbia.edu

\Address

\Address
\Address
\EndAddress
\section{introduction}

\indent De Jong--Oort's purity theorem \cite[Theorem
4.1]{dejong-oort} states that for an $F$-crystal over a noetherian
scheme $S$ of characteristic $p$ the associated Newton polygons over
all points of $S$ are constant if this is true outside a subset of
codimension bigger than 1. This theorem has been strengthened and
generalized by Vasiu \cite{Vasiu}, who has shown that each stratum
of the Newton polygon stratification defined by an $F$-crystal over
any reduced, not necessarily noetherian $\mathbb{F}_p$-scheme $S$ is
an affine $S$-scheme. In the case of a family of $p$-divisible
groups, alternative proofs of the purity have been given by Oort
\cite{Oort} and Zink \cite{Zink}. In this paper we show an
improvement, which implies that for an $F$-crystal over a noetherian
scheme $S$ the Newton polygons over all points have a common break
point if this is true outside a subset of codimension bigger than 1.
As to a stronger statement analogous to that in Vasiu's paper, our
method does not apply. The main result is the following theorem.

\begin{thm}

Let $S$ be a locally Noetherian scheme of characteristic $p$ and
$\mathcal{E}$ be an $F$-crystal over $S$. Fix $s\in S$. If there
exists an open neighborhood $U$ of $s$ in $S$ such that the Newton
polygons $NP(\mathcal{E})_x$ over all points $x\in U\backslash\{s\}$
have a common break point, then either codim$(\{s\}^-, U)\leq 1$ or
$NP(\mathcal{E})_s$ has the same break point.\label{mainresult}
\end{thm}

The following example explains how Theorem ~\ref{mainresult}
improves de Jong-Oort purity theorem. Look at $Figure\ 1$. Consider
the spectrum of some local Noetherian integral domain of dimension
$2$ and characteristic $p$. Then we ask: does there exist an
$F$-crystal such that the associated Newton polygon over the closed
point is $\xi$, over a finite number of points of codimension $1$ is
$\gamma$ and over each of all other points is $\eta$? Theorem
~\ref{mainresult} tells us that the answer is negative, while it
cannot be easily seen from \cite[Theorem 4.1]{dejong-oort}.

In our main theorem, the condition on ``one of the break points"
cannot be generalized to an arbitrary point of the Newton polygon
which is not a break point. Consider a family of elliptic curves $f:
\mathcal{X}\rightarrow S$, where $S$ is a curve over a field $k$ of
positive characteristic. Look at $Figure\ 2$. Assume that all the
fibers of $f$ are ordinary except over a closed point $0\in S$. Then
$Figure\ 2$ shows all the Newton polygons associated to the family
of abelian surfaces $\mathcal{X}\times_k\mathcal{X}\rightarrow
S\times_kS$. Namely, over the special point $(0,0)$ the associated
Newton polygon is $\xi$; over each point in $\{0\}\times S\cup
S\times \{0\}$, it is $\gamma'$; over each of all other points, it
is $\eta'$. We see that outside the one-point set $\{(0,0)\}$ of
codimension 2, the Newton polygons have a common point $P$.

\setlength{\unitlength}{1.2cm}
\begin{picture}(10,4)
\put(0,0.5){\line(1,0){3}} \put(3,0.5){\line(0,1){3}}
\put(0,0.5){\line(1,0){1}} \put(1,0.5){\line(1,1){1}}
\put(2,1.5){\line(1,2){1}} \put(1,0.5){\line(2,3){2}}
\put(0,0.5){\line(1,1){3}} \put(2.2,1.5){$\eta$}
\put(1.7,1.5){$\gamma$} \put(1.2,1.5){$\xi$}\put(1,0.2){Figure 1}
\put(5,0.5){\line(2,1){4}} \put(9,0.5){\line(0,1){2}}
\put(6,0.5){\line(2,1){2}} \put(8,1.5){\line(1,1){1}}
\put(7,0.5){\line(1,1){1}} \put(5,0.5){\line(1,0){4}}
\put(7,0.2){Figure 2} \put(6.4,1){$\xi'$} \put(7.3,1){$\gamma'$}
\put(7.8,1){$\eta'$} \put(7.9,1.6){$P$}
\end{picture}

This paper is organized as follows. In Section 2, we review some
facts about F-crystals before showing that if the Newton polygon of
an F-crystal over a field has a break point $(1,m)$, then there
exists a unique subcrystal of rank $1$ and slope $m$. In Section 3,
we describe the kernel of
$\text{Gal}(\overline{K}/K)\rightarrow\pi_1(X,\overline{\eta})$ as
the normal subgroup generated by local kernels (Proposition
~\ref{Kernel}) and particularly obtain another description when $X$
is the spectrum of a discrete valuation ring (Corollary
~\ref{reducekernel}). In Section 4, we define the Galois
representation associated to an F-crystal, and discuss the
relationship between its ramification property and Newton slopes
(see proposition ~\ref{Trivialrepr}). Section 5 contains the proof
of Theorem ~\ref{mainresult}. The proof essentially follows the
proof of \cite[Theorem 4.1]{dejong-oort}, yet is more accessible
because the relationship between the ramification property of the
representation and the Newton slopes has been clarified.

The author would like to thank her advisor, Aise Johan de Jong,
without whose patient and effective instruction this paper would not
have come into existence. Also the author owes a lot to the referees
who have made many suggestions and particularly pointed out a
significant improvement of the main theorem, which had originally
treated the first few break points of the Newton polygon instead of
any break point .

\section{Results on F-crystals}

\subsection{Conventions.} In this paper, $k$ always denotes a field of
characteristic $p$, where $p$ is a prime number; $\overline{k}$
denotes an algebraic closure of $k$; $S$ denotes a connected scheme
of characteristic $p$. We use the term $crystal$ to mean a crystal
of finite locally free $\mathcal{O}_{cris}-$modules. See \cite[Page
226]{Berthelot}. Here $\mathcal{O}_{cris}$ denotes the structure
sheaf on the category CRIS$(S/\text{Spec }\mathds{Z}_p)$ (big
crystalline site of $S$). If $T\rightarrow S$ is a morphism, then we
use $\mathcal{E}|_T$ to denote the pullback of $\mathcal{E}$ to
CRIS$(T/\text{Spec }\mathds{Z}_p)$. For a crystal $\mathcal{E}$, we
denote by $\mathcal{E}^{(n)}$ the pullback of $\mathcal{E}$ by the
$n$th iterate of the Frobenius endomorphism of $S$. An $F$-$crystal\
over\ S$ is a pair $(\mathcal{E}, F)$, where $\mathcal{E}$ is a
crystal over $S$ and $F:\mathcal{E}^{(1)}\rightarrow\mathcal{E}$ is
a morphism of crystals. We usually denote an $F$-crystal by
$\mathcal{E}$, with the map $F$ being understood. Recall that
$\mathcal{E}$ is a $nondegenerate\ F$-$crystal$ if the kernel and
cokernel of $F$ are annihilated by some power of $p$, see
\cite[3.1.1]{Rivano}. All $F$-crystals in this paper will be
nondegenerate.

A perfect scheme $S$ in characteristic $p$ is a scheme such that the
Frobenius map $(-)^p:\mathcal{O}_S\rightarrow\mathcal{O}_S$ is an
isomorphism. A crystal over a perfect scheme $S$ is simply given by
a finite locally free sheaf of $W(\mathcal{O}_S)$-modules (see
\cite[Page 141]{Katz}).

\subsection{} Suppose that $S=\text{Spec }k$. Choose
a Cohen ring $\Lambda$ for $k$, and let $\sigma:
\Lambda\rightarrow\Lambda$ be a lift of $Frobenius$ on $k$. By
\cite[Proposition 1.3.3]{Berthelot-Messing}, we know that an
$F$-crystal $\mathcal{E}$ over $k$ is given by a triple
$(M,\nabla,F)$ over $\Lambda$, where $M$ is a finite free
$\Lambda$-module of rank $r$, $\nabla$ is an integrable,
topologically quasi-nilpotent connection, and $F$ is a horizontal
$\sigma$-linear self-map of $M$. \label{modandcry}

\subsection{} Let $k^{pf}$ be the perfect closure of $k$. Note that under the
identification $k^{pf}=\varinjlim(k\rightarrow k\rightarrow\ldots)$,
and by \cite[Chapter II, Prop 10]{Serre}, we obtain
$$W(k^{pf})=p\text{-adic completion of }\varinjlim(\Lambda\xrightarrow{\sigma}\Lambda\xrightarrow{\sigma}\ldots).$$
Furthermore $\sigma$ can be extended to an endomorphism of
$W(k^{pf})$, which is a lift of Frobenius on $k^{pf}$, still denoted
by $\sigma$. Thus we get an injection $\Lambda\rightarrow W(k^{pf})$
compatible with $\sigma$. \label{cohenwitt}

The pullback $\mathcal{E}|_{\text{Spec}(k^{pf})}$ of $\mathcal{E}$
corresponds to the pair $(M\otimes_\Lambda
W(k^{pf}),F\otimes\sigma)$. According to \cite[1.3]{Katz}, we can
describe Newton slopes associated to
$\mathcal{E}|_{\text{Spec}(k^{pf})}$ as follows. Choose an algebraic
closure $\overline{k}$ of $k^{pf}$ and some positive integer $N$
divisible by $r!$. Consider the valuation ring
$R=W(\overline{k})[X]/(X^N-p)=W(\overline{k})[p^{1/N}]$ and denote
its fraction field by $K$. We extend $\sigma$ to an automorphism of
$R$ by requiring that $\sigma(X)=X$. Then by Dieudonn\'{e}( cf
\cite{Manin}), $M\otimes_{W(\overline{k})} K$ admits a $K$-basis
$e_1,\ldots,e_r$ such that $(F\otimes\sigma)(e_i)=p^{\lambda_i}e_i$
and $0\leq\lambda_1\leq\ldots\leq\lambda_r$. These $r$ rational
numbers are defined to be the \emph{Newton slopes} of $(M,F)$ or
$\mathcal{E}$.

For each $\lambda$, we define $\text{mult}(\lambda)$ as the number
of\ times  $\lambda$ occurs among $\{\lambda_1,\ldots,\lambda_r\}$.
By Dieudonn\'{e} again, the product $\lambda\
\text{mult}(\lambda)\in \mathds{Z}_{\geq0}$  for each $\lambda$. The
Newton Polygon of $(M,F)$ is a polygonal chain consisting of line
segments $S_1,\ldots,\ S_r$, where $S_i$ connects the two points
$(i-1,\lambda_1+\ldots+\lambda_{i-1})\ \text{and}\
(i,\lambda_1+\ldots+\lambda_i)$. The points at which the Newton
polygon changes slope are called \emph{break points}.

We now turn to an $F$-crystal $\mathcal{E}$ over an arbitrary
$\mathds{F}_p$-scheme $S$. For every point $s\in S$, let
$s:\text{Spec }k(s)\rightarrow S$ be the natural map. We can assign
to $s$ the Newton polygon associated with $\mathcal{E}|_{\text{ Spec
}k(s)}$, denoted by $NP(S,\mathcal{E})_s$ or $NP(\mathcal{E})_s$.

The following result about the existence of some special subcrystal
will be significant in proving the theorem.
\begin{prop}\label{subcrystal}
Let $(\mathcal{E},F)$ be a crystal over $S=\text{Spec}(k)$. If the
first break point of $NP(S,\mathcal{E})$ is $(1,m)$, where
$m\in\mathds{Z}_{\geq0}$, then it has a unique subcrystal
$\mathcal{E}_1\subset\mathcal{E}$ of rank $1$ and slope $m$.
\end{prop}
\begin{proof}
$Proof$: Let $(M,\nabla, F)$ be the triple corresponding to the
 crystal $\mathcal{E}$ by ~\ref{modandcry}. Then the existence of the required subcrystal
is equivalent to the existence of a unique $\Lambda$-submodule $M_1$
of rank 1 and slope $m$, preserved by the action of $\nabla$. We
will first find a submodule of rank $1$ and slope $m$, then show it
is preserved by $\nabla$. The uniqueness of such a submodule follows
from the fact that the lowest slope is of multiplicity $1$.

Choose $\overline{k}\supset k^{pf}\supset k$. From ~\ref{cohenwitt},
we have a faithfully flat homomorphism $\Lambda\xrightarrow{i}
W(\overline{k})$. Let $\overline{M}=M\otimes_\Lambda
W(\overline{k})$. By \cite[Theorem 2.6.1]{Katz}, there is an isogeny
$\psi:\overline{M}\rightarrow N$, where $\frac{1}{p^m}F_N:
N\rightarrow N$ is a $\sigma$-linear self-map. By \cite[Theorem
1.6.1]{Katz}, $N$ has a unique free submodule $N_1$ of rank $1$ and
slope $m$ such that $N/N_1$ is free as a $W(\overline{k})$-module.
Let $\overline{M}_1=\psi^{-1}(N_1)$. It is clear that
$\overline{M}_1$ is a module of rank 1 and slope $m$, and
$\overline{M}_2=\overline{M}/\overline{M}_1$ is a free
$\Lambda$-module of rank $r-1$ and slopes $>m$.

Since $\psi$ is an isogeny, there exists some $D\in\mathds{Z}_{>0}$
such that $p^D\psi^{-1}(N)\subset\overline{M}$. As for every
$\nu>0$,
$p^D(\frac{F}{p^m})^\nu=p^D\psi^{-1}(\frac{F_N}{p^m})^\nu\psi$, thus
$p^D(\frac{F}{p^m})^\nu:\overline{M}\rightarrow \overline{M}$.
Actually we can choose $D_\nu\in [0,D]$ such that the matrix of
$\overline{f^\nu}=p^{D_\nu}(\frac{F}{p^m})^\nu$ mod $p$ does not
vanish. Let $f^\nu=p^{D_\nu}(\frac{F}{p^m})^\nu: M\rightarrow M$,
then $f^\nu$ mod $p$ does not vanish either. Since the Newton slopes
of $\overline{M}_2$ are greater than $m$, according to
\cite[1.4.3]{Katz}, for each $n>0$ there exists $c_n> 0$ such that
$\overline{f^\nu}(\overline{M}_2)\subset p^n\overline{M}_2$ for all
$\nu\geq c_n$.  Let $\overline{f^\nu_n}:
\overline{M}/p^n\overline{M}\rightarrow
\overline{M}/p^n\overline{M}$. Then
$\text{Im}(\overline{f^\nu_n})\subset\overline{M}_1/p^n\overline{M}_1$.
Let $\overline{E_n^\nu}=<\text{Im}(\overline{f^\nu_n})>$.  Note that
$<G>$ is denoted as the smallest $R/p^nR$-submodule of $M$
containing $G$, where $R$ is a discrete valuation ring with $p$ as
its uniformizer, $M$ is a finite free $R/p^nR$-module and $G\subset
M$ a subset.

Let $f^\nu_n: M/p^nM\rightarrow M/p^nM$, $E_n^\nu=<\text{Im}
(f^\nu_n)>$,  and $E_n=\cap_{\nu\geq c_n}E_n^\nu$. We get $
\overline{E_n^\nu}=E_n^\nu\otimes_\Lambda W(\overline{k})$, and
$\overline{E}_n=E_n\otimes_\Lambda W(\overline{k})=\cap_{\nu\geq
c_n}\overline{E_n^\nu}$. By the above argument, when $\nu\geq c_n$,
$\overline{E_n^\nu}\simeq\overline{M}_1/p^n\overline{M}_1$, a free
$W(\overline{k})/p^nW(\overline{k})$-module of rank $1$.

As $\Lambda\xrightarrow{i} W(\overline{k})$ is faithfully flat and
$\overline{E_n^\nu}=E_n^\nu\otimes_\Lambda W(\overline{k})$ is a
free $W(\overline{k})/p^nW(\overline{k})$-module of rank $1$ for
$\nu\geq c_n$, then $E_n^\nu$ is a free $\Lambda/p^n\Lambda$-module
of rank 1, hence so is $E_n$. Also the surjectivity of
$\overline{E}_{n+1}\rightarrow\overline{E}_n$ implies the
surjectivity of $E_{n+1}\rightarrow E_n$. Let $M_1=\varprojlim_{n>0}
E_n$, it is easy to see that $M_1$ is a free $\Lambda$-module. Since
$M_1\otimes_\Lambda W(\overline{k})=\overline{M_1}$ has slope $m$,
so does $M_1$.

Now we show $\nabla(M_1)\subset M_1\otimes_\Lambda\Omega_\Lambda$.
Here $\Omega_\Lambda = \varprojlim_n
\Omega^1_{(\Lambda/p^n\Lambda)/\mathds{Z}}$ is the $p$-adic module
of differentials. Let $\{e_1, ... , e_r\}$ be a basis of $M$ and
$e_1\in M_1$. Suppose $\nabla(e_1)=\sum_{i=1}^re_i\otimes\eta^i$. We
need to show $\eta^i=0$ for $i>1$. As $F^{\nu}$ is a horizontal
$\sigma^\nu$-linear self map for $\nu>0$, it exchanges with $\nabla$
in the following sense: $\widetilde{F^\nu}\circ\nabla=\nabla\circ
F^\nu$, where $\widetilde{F^\nu}=F^\nu\otimes\widetilde{\sigma^\nu}$
is the endomorphism of $M\otimes_\Lambda \Omega_{\Lambda}$ and
$\widetilde{\sigma^\nu}$ is the map $\Omega_{\Lambda}\rightarrow
\Omega_{\Lambda}$ given by $\alpha d\beta\mapsto
\sigma^\nu(\alpha)d\sigma^\nu(\beta)$. Then from
$\widetilde{F^\nu}\circ\nabla(e_1)=\nabla\circ F^\nu(e_1)$ we deduce
that
$$p^{m\nu}\mu_\nu\sum_{i>1}e_i\otimes\eta^i
=\sum_{i>1}F^\nu(e_i)\otimes\widetilde{\sigma^\nu}(\eta^i) \text{
mod } M_1\otimes\Omega_{\Lambda},$$ where $\mu_\nu\in\Lambda^*$. By
\cite[1.4.3]{Katz}, $F^\nu(M/M_1)\subset p^{m\nu+1}(M/M_1)$ for
$\nu>>0$. By comparing terms before $e_i$ in the above equation, we
get $\eta^i\in p\Omega_{\Lambda}$. Replace $\eta^i$ by
$p\widetilde{\eta^i}$ on the right side, then we get $\eta^i\in
p^2\Omega_{\Lambda}$. By repeating, $\eta^i\in p^n\Omega_{\Lambda}$
for every $n$. By \cite[1.3.1 Proposition]{Berthelot-Messing},
$\Omega_{\Lambda}$ is a free $\Lambda$-module. Then $\eta^i=0$ for
$i>1$. Hence $M_1$ is preserved by $\nabla$.
\end{proof}

\begin{rem}
The proposition can be generalized in the following way: Let
$(\mathcal{E},F)$ be a crystal over $S=\text{Spec}(k)$. If the first
break point of $NP(S,\mathcal{E})$ is $(\mu_1,\mu_1\lambda_1)$,
where $\lambda_1$ is the lowest Newton slope and $\mu_1$ is its
multiplicity, then there is a unique sub-crystal
$\mathcal{E}'\subset\mathcal{E}$ of rank $\mu_1$ with its Newton
slopes all equal to $\lambda_1$.

Applying the lemma to $(\wedge^{\mu_1}\mathcal{E},\wedge^{\mu_1}F)$,
we obtain a subcrystal $\mathcal{E}_1$ of
$\wedge^{\mu_1}\mathcal{E}$. To see that $\mathcal{E}_1$ is of the
form $\wedge^{\mu_1}\mathcal{E}'$ for some subcrystal
$\mathcal{E}'\subset\mathcal{E}$, we need to use the Pl\"{u}cker
coordinate and check if $\mathcal{E}_1$ satisfies the Pl\"{u}cker
equations. By extending the scalars to the fraction field $K$ of
$W(\overline{k})[X]/(X^N-p)$ for some proper $N$, we obtain that
$\mathcal{E}\otimes K$ admits a $K$-basis over which the matrix of
$F$ is diagonalized, hence the unique subcrystal
$\mathcal{E}_1\otimes K$ of rank 1 and slope m satisfies the
Pl\"{u}cker equations, and so does $\mathcal{E}_1$.

\end{rem}

\section{Facts about Fundamental Groups}

\subsection{} Let $X$ be a noetherian normal integral scheme with
its generic point $\eta$. Let $\overline{\eta}$ be a geometric point
over $\eta$. By \cite[Expos\'{e} V, Proposition 8.2]{Grothendieck},
the canonical map
$\phi:\text{Gal}(\overline{K}/K)\rightarrow\pi_1(X,\overline{\eta})$
is surjective, and the kernel is $\text{Gal}(\overline{K}/M)$, where
$\overline{K}$ is some algebraic closure of the fraction field $K$
of $X$ and $M$ is the union of all finite subextensions $K\subset
L(\subset\overline{K})$ such that $L$ is unramified over $X$, which
means that the normalization of $X$ in $L$ is unramified over $X$.
\label{kerofgalois}

This section focuses on describing the kernel of $\phi$ in terms of
local kernels. Assume that the completion
$\widehat{\mathcal{O}}_{X,x}$ of the local ring $\mathcal{O}_{X,x}$
at every point $x$ is an integral domain.  Denote the fraction field
and residue field of $\widehat{\mathcal{O}}_{X,x}$ by $K_x$ or
$k(x)$ respectively. Let $\overline{K}_x$ be an algebraic closure of
$K_x$ and $\overline{\eta}_x$ be the geometric point defined by
$\text{Spec }\overline{K}_x\rightarrow\text{Spec
}\widehat{\mathcal{O}}_{X,x}$. Fix some injection
$\omega:\overline{K}\rightarrow\overline{K}_x$ such that we have the
following commutative diagram:
$$\xymatrix{K\ar[r]\ar[d]&\overline{K}\ar[d]^\omega \\K_x \ar[r]
&\overline{K}_x }$$

Thus we have maps of Galois groups depending on $\omega$:
$$\psi_x: \text{Gal}(\overline{K}_x/K_x)\rightarrow
\text{Gal}(\overline{K}/K),\ \ \phi_x:
\text{Gal}(\overline{K}_x/K_x)\rightarrow\pi_1(\text{Spec
}\widehat{\mathcal{O}}_{X,x},\overline{\eta}_x)$$

\begin{prop}
Let $X$ be a noetherian normal integral scheme with $K$ as its
function field. Let $\phi,\ \phi_x$ and $\psi_x$ be the same as
above. If assuming that the completion $\widehat{\mathcal{O}}_{X,x}$
of the local ring $\mathcal{O}_{X,x}$ at each closed point $x\in X$
is a normal domain and that the same condition holds for the
normalization of $X$ in every finite separable extension $L/K$, then
$\text{Ker }\phi=H$, where $H$ is the normal closed subgroup of
$\text{Gal}(\overline{K}/K)$ generated by $\{\psi_x(\text{Ker
}\phi_x)|\ x\text{ is a closed point of } X\}$.\label{Kernel}
\end{prop}

Note that if moreover $X$ is an excellent scheme, the conditions on
the local ring $\mathcal{O}_{X,x}$ at every closed point $x\in X$
are satisfied. In the following, let $x\in X$ be a closed point and
$L/K$ be a finite separable subextension in $\overline{K}$ if no
other description is given.

\subsection{}Let $\widetilde{X}$ be the
normalization of X  and $\widetilde{\mathcal{O}}_{X,x}$ be the
integral closure of $\mathcal{O}_{X,x}$ in $L$. By \cite[Chapter I,
Proposition 8]{Serre}, $\widetilde{X}\rightarrow X$ is a finite
morphism, and $\widetilde{\mathcal{O}}_{X,x}$ is a finitely
generated $\mathcal{O}_{X,x}$-module. Let $\{x_i\in\widetilde{X},\
i\in I\}$ be the set of points over $x$. Since
$\widetilde{\mathcal{O}}_{X,x}$ is a semilocal ring and a finite
$\mathcal{O}_{X,x}$-module,  then by \cite[Chapter I, Theorem
4.2]{Milne},
$\widetilde{\mathcal{O}}_{X,x}\otimes_{\mathcal{O}_{X,x}}\widehat{\mathcal{O}}_{X,x}=\prod_{i\in
I}\widehat{\mathcal{O}}_{\widetilde{X},x_i}$, where
$\widehat{\mathcal{O}}_{\widetilde{X},x_i}$ is the completion of the
local ring of $x_i\in\widetilde{X}$, and
$\widehat{\mathcal{O}}_{\widetilde{X},x_i}$ is a finite
$\widehat{\mathcal{O}}_{X,x}$-algebra. Thus we have the following
cartesian diagram:

$$\xymatrix{\text{Spec }\widehat{\mathcal{O}}_{X,x}\ar[d]&
\text{Spec }\prod_{i\in I}\widehat{\mathcal{O}}_{\widetilde{X},x_i}
\ar[l] \ar[d]
&\text{Spec }\widehat{\mathcal{O}}_{\widetilde{X},x_i} \ar[d]\ar[l]\\
\text{Spec }\mathcal{O}_{X,x} & \text{Spec
}\widetilde{\mathcal{O}}_{X,x}\ar[l] & \text{Spec
}\mathcal{O}_{\widetilde{X},x_i} \ar[l]}$$\label{diagram}

\subsection{}By \cite[Chapter I, Proposition 3.5]{Milne}, $L$ is
unramified over $X$ if and only if $\Omega_{\widetilde{X}/X}^1=0$.
As the branch locus where $\Omega_{\widetilde{X}/X}^1\neq0$ is a
closed subset, and $\Omega^1$ behaves well with respect to base
change, then $L$  is unramified over $X$ if and only if
 $$\Omega^1_{
\text{Spec }\mathcal{O}_{\widetilde{X},x_i}/\text{Spec
}\mathcal{O}_{X,x}}=0, \text{ for every closed point
}x_i\in\widetilde{X} \text{ over } x\in X.$$ As
$\widehat{\mathcal{O}}_{\widetilde{X},x_i}$ is a faithfully flat
$\mathcal{O}_{\widetilde{X},x}$-module, this is also equivalent to

 \begin{equation}
 \Omega^1_{
\text{Spec }\widehat{\mathcal{O}}_{\widetilde{X},x_i}/\text{Spec
}\widehat{\mathcal{O}}_{X,x}}=0, \text{ for every closed point
}x_i\in\widetilde{X} \text{ over } x\in X.\label{globalunr}
\end{equation}

\subsection{}
Let $L. K_x=\omega(L)K_x$. It is clear that $L.K_x$ is a separable
extension of $K_x$. Actually $L.K_x$ is the fraction field of
$\widehat{\mathcal{O}}_{\widetilde{X},x_i}$ for some $i$. Since by
base change the diagram in ~\ref{diagram} from $\text{Spec
}\widehat{\mathcal{O}}_{X,x}\rightarrow\text{Spec
}\mathcal{O}_{X,x}$ to $\text{Spec } K_x\rightarrow\text{Spec }K$,
we have $L\otimes_KK_x\simeq\prod_{i\in
I}Frac(\widehat{\mathcal{O}}_{\widetilde{X},x_i})$; by choosing
$\omega$, one has to choose $L\rightarrow
Frac(\widehat{\mathcal{O}}_{\widetilde{X},x_i})$. Hence $L/K$ is
fixed by $H$ if and only if for every closed point $x\in X$, there
exists some $x_i$ over $x$ such that
$Frac(\widehat{\mathcal{O}}_{\widetilde{X},x_i})$ is unramified over
$\text{Spec }\widehat{\mathcal{O}}_{X,x}$; by assumption this is
equivalent to $\text{Spec
}\widehat{\mathcal{O}}_{\widetilde{X},x_i}\rightarrow \text{Spec
}\widehat{\mathcal{O}}_{X,x}$ being unramified. Hence $L/K$ is fixed
by $H$ if and only if
\begin{equation}
\Omega^1_{ \text{Spec
}\widehat{\mathcal{O}}_{\widetilde{X},x_i}/\text{Spec
}\widehat{\mathcal{O}}_{X,x}}=0, \text{ for every closed point }x\in
X, \text{ some }x_i\in\widetilde{X} \text{ over } x.\label{localunr}
\end{equation}

\subsection{} Assume that $L/K$ is a finite Galois extension. Let $U\subset
X$ be an affine neighborhood of $x$ and $\widetilde{U}$ be its
normalization in $L$. By \cite[Chapter 2, 5.E]{Matsumura}, for two
given points $x_i,x_j\in\widetilde{U}$ over $x\in U$, there exists a
$U$-automorphism of $\widetilde{U}$ mapping $x_i$ to $x_j$, hence
$\Omega^1_{ \text{Spec }\mathcal{O}_{\widetilde{U},x_i}/\text{Spec
}\mathcal{O}_{U,x}}\simeq \Omega^1_{ \text{Spec
}\mathcal{O}_{\widetilde{U},x_j}/\text{Spec }\mathcal{O}_{U,x}}$. It
follows that $\widetilde{X}\rightarrow X$ is unramified at some
point over $x$ if and only if it is unramified at every point over
$x$.\label{transitive}

\begin{proof}
Proof of Proposition~\ref{Kernel}:  Let $N$ be the subfield of
$\overline{K}$  fixed by $H$. Since both $H$ and $\text{Ker }\psi$
are normal subgroups, then it suffices to show that $N=M$. Assume
$L/K$ is a finite Galois subextension. From discussions in
~\ref{transitive}, the conditions ~\ref{globalunr} and
~\ref{localunr} are equivalent. Then we have $L\subset
M\Leftrightarrow L\subset N$, hence $M=N$.
\end{proof}

We will also need the following facts about Galois groups.

\begin{clm}
Let $(K,\nu)$ be a henselian field with a nonarchimedean valuation
$\nu$. Let $(K_\nu,\nu)$ be its completion. Denote by
$\overline{K}$(resp. $\overline{K}_{\nu}$) the algebraic closure of
$K$(resp. $K_\nu$). Then the homomorphism $\text{Gal
}(\overline{K}_{\nu}/K_{\nu})\rightarrow \text{Gal
}(\overline{K}/K)$ is surjective.\label{surjective}

\end{clm}

\begin{fact}
Let $l/k$ be an algebraic extension. Let $K=k((t))$,
$\widehat{L}=l((t))$, and $L=\bigcup m((t))$ where $m$ runs over all
finite subextensions of $k$ in $l$. There is an obvious valuation
$\nu$ on $L$ and $\widehat{L}$ by sending $t^n$ to $n$. It is clear
that $\widehat{L}$ is the completion of $L$. As the valuation ring
of $L$ is $R=\bigcup m[[t]]$, from Definition(6.1) in \cite[Chapter
II]{Neukirch}, $L$ is a henselian field. Thus
$\text{Gal}_{\widehat{L}}\rightarrow\text{Gal}_{L}$ is surjective.

\label{closure}
\end{fact}

\begin{cor}

Let $R$ be a discrete valuation ring of characteristic $p$ with
fraction field $K$ and residue field $k$. let $s\in\text{Spec }R$ be
the closed point and $\overline{\eta}$ be a geometric point over the
generic point. Let $R_s$ be the completion of $R$, which is of the
form $k[[t]]$. Then the kernel of the canonical homomorphism
$\xymatrix{\text{Gal}_K\ar[r]^\phi&\pi_1(\text{Spec }R,
\overline{\eta})}$ is the normal subgroup of $\text{Gal}_K$
generated by the image of the composition
$\xymatrix{\text{Gal}_{\overline{k}((t))}\ar[r]&\text{Gal}_{k((t))}\ar[r]^{\psi_s}
& \text{Gal}_K}$. \label{reducekernel}
\end{cor}
\begin{proof}
$Proof$: Since $\text{Spec }R$ satisfies the assumption in
Proposition ~\ref{Kernel},  $\text{Ker }\phi$ is generated by
$\psi_s(\text{Ker }\phi_s)$. Apply Fact ~\ref{closure} to the case
when $l=k^{sep}$. Note that $L$ is the maximal unramified algebraic
extension of $X=\text{Spec }k[[t]]$ in the sense of
~\ref{kerofgalois}, and hence  $\text{Ker }\phi_s$
in~\ref{kerofgalois} is the normal subgroup generated by the image
of $Gal_{k^{sep}((t))}\twoheadrightarrow\text{Gal}_L\rightarrow
\text{Gal}_{k((t))}$. Apply Fact~\ref{closure} to the field
extension $\overline{k}/k^{sep}$ to see that
$\text{Gal}_{\overline{k}((t))}\rightarrow\text{Gal}_{k^{sep}((t))}$
is surjective. In conclusion, $\text{Ker }\phi_s$ is the image of
$\text{Gal}_{\overline{k}((t))}\rightarrow\text{Gal}_{k((t))}$.
\end{proof}

\section{Galois Representations Associated to F-crystals of Rank 1}

\subsection{} Consider an $F$-crystal $\mathcal{E}$ of rank $1$
and slope $m$ over $k$. Let $(M,\nabla,F)$ over $\Lambda$ be the
triple  defining the crystal $\mathcal{E}$. If $\{e\}$ is chosen to
be the basis of $M$, then $F(e)=p^m\mu e$, where $\mu$ is a unit in
$\Lambda\subset W(k^{pf})$. By~\ref{cohenwitt}, there exists some
unit $\alpha\in W(\overline{k})$ such that
$F(e\otimes\alpha)=p^me\otimes\alpha$, i.e.
$\sigma(\alpha)\mu=\alpha$. As every $g\in
\text{Gal}(\overline{k}/k)=\text{Gal}(\overline{k}/k^{pf})$ can be
uniquely lifted as a $W(k^{pf})$-automorphism of $W(\overline{k})$,
it is easy to show that $g(\alpha)\alpha^{-1}\in\mathds{Z}_p^*$.
Thus we get a continuous homomorphism $\rho:
\text{Gal}(\overline{k}/k)\rightarrow \mathds{Z}_p^*$ by sending $g$
to $g(\alpha)\alpha^{-1}$. \label{defnofrepr}

\begin{defn}
Let $\mathcal{E}$ be an $F$-crystal over a noetherian integral
scheme $X$ of characteristic $p$. Let $K$ be the fraction field of
$X$ and $\eta$ be the generic point. Assume that the first break
point of $NP(X,\mathcal{E})_\eta$ is $(1,m)$, where
$m\in\mathds{Z}_{\geq0}$. Then by Proposition~\ref{subcrystal},
there exists a unique subcrystal
$\mathcal{E}_1\subset\mathcal{E}_\eta$ of rank $1$ and slope $m$. By
the above discussion we obtain from the crystal $\mathcal{E}_1$  a
continuous homomorphism $\rho: \text{Gal}(\overline{K}/K)\rightarrow
\mathds{Z}_p^*$. We call it \emph{the Galois representation
associated to $\mathcal{E}$}, or \emph{the associated representation
of $\mathcal{E}$}. \label{formaldefinition}
\end{defn}

\subsection{}Let $X$ and $Y$ be noetherian integral schemes.
Let $f:X\rightarrow Y$ be a morphism mapping the generic point of
$X$ to the generic point of $Y$. Assume that $\mathcal{E}$ is a
crystal over $Y$ satisfying the assumption of the definition. Then
the representation associated to $\mathcal{E}|_X$ is the composition
$\text{Gal}_{K(X)}\rightarrow \text{Gal}_{K(Y)}\rightarrow
\mathds{Z}_p^*$.\label{compatiblerepr}

\begin{lem} Let $(\mathcal{E}, F_1)$ and $(\mathcal{E}', F_2)$ be
two $F$-crystals over a noetherian integral scheme $X$ satisfying
the assumptions in Definition~\ref{formaldefinition}. If there
exists an isogeny $\psi:\mathcal{E}\rightarrow\mathcal{E}'$, then
their associated representations are identical. \label{isogenous}
\end{lem}
\begin{proof}
$Proof$: Let $\mathcal{E}_1$ (resp. $\mathcal{E}'_1$) be the
subcrystal of $(\mathcal{E})_\eta$ (resp. $(\mathcal{E}')_\eta$)
obtained in Proposition ~\ref{subcrystal}. As $\psi\circ
F_1=F_2\circ\psi$, then $\psi(\mathcal{E}_1)\subset\mathcal{E}'_1$.
Actually we can choose a basis $e_i$ of $\mathcal{E}_1$ (resp.
$\mathcal{E}'_1$) so that $\psi(e_1)=p^ne_2$ for some $n\in N$, and
$F_i{e_i}=p^m\mu e_i$ for some unit $\mu\in\Lambda$. Then it is
obvious that the Galois representations associated to $\mathcal{E}$
and $\mathcal{E}'$ are identical.
\end{proof}

\begin{lem}
Let $\mathcal{E}$ be an $F$-crystal of rank 1 and slope
$m\in\mathds{Z}_{\geq0}$ over $S=\text{Spec }k$, where $k$ is a
field of characteristic $p$. If the associated representation is
trivial, then $\mathcal{E}$ is a trivial crystal, i.e. there exists
some basis $\{e\}$ of $\mathcal{E}$ such that $F(e)=p^me$, and
$\nabla(e)=0$. \label{trivialcrystal}
\end{lem}
\begin{proof}
$Proof$: Let $e$ be a basis of $\mathcal{E}$, and $F(e)=p^m\mu e$.
From ~\ref{defnofrepr}, there is some unit $\alpha\in
W(\overline{k})$ such that $\sigma(\alpha)\mu=\alpha$; the
associated representation is trivial if and only if the unit
$\alpha\in W(k^{pf})$. It suffices to show that $\alpha\in\Lambda$,
and $\nabla(e)=0$ follows automatically.

Let $U^n(k)=1+p^n\Lambda$ and $U^n(k^{pf})=1+p^nW(k^{pf})$. First
choose $\mu\in U^1(k)$ and $\alpha\in U^1(k^{pf})$. Considering
$\sigma(\alpha)\mu=\alpha(\text{ mod }p)$, we have
$\overline{\alpha}^p \overline{\mu}=\overline{\alpha}$, where
$\overline{\alpha}=(\alpha\text{ mod }p) \in k^{pf}$. Moreover, the
equation implies that $\overline{\alpha}$ is separable over $k$, and
hence $\overline{\alpha}\in k$. Choose $\gamma_0\in\Lambda$ such
that $\gamma_0(\text{ mod }p)=\overline{\alpha}$. Replace the basis
$e$ by $\gamma_0e$, then replace $\mu$ by
$\sigma(\gamma_0)\mu\gamma_0^{-1}$ and $\alpha$ by
$\alpha\cdot\gamma_0^{-1}$. Then $\sigma(\alpha)\mu=\alpha$ still
holds, and $\mu\in U^1(k)$, $\alpha\in U^1(k^{pf})$.

The induction step: Assume $\mu_{n-1}\in U^n(k)$, $\alpha_{n-1}\in
U^n(k^{pf})$, and $\sigma(\alpha_{n-1})\mu_{n-1}=\alpha_{n-1}$. It
suffices to show that there exists some $\gamma_n\in U^n(k)$ such
that $\gamma_n=\alpha_{n-1}(\text{ mod }p^{n+1})$. Write
$\mu_{n-1}=1+p^n\nu_n$, $\alpha_{n-1}=1+p^n\delta_n$ for some
$\nu_n\in\Lambda$, $\delta_n\in W(k^{pf})$. By assumption we have
$$\sigma(\delta_n)+\nu_n=\delta_n (\text{ mod }p)\text{ or } \overline{\delta}_n^p
+\overline{\nu}_n=\overline{\delta}_n$$

As $\overline{\delta}_n\in k^{pf}$, and since the above equation
implies that it is separable over $k$,  $\overline{\delta}_n\in k$.
Hence we can choose $\gamma_n\in U^n(k)$ such that
$\gamma_n=\alpha_{n-1}(\text{ mod }p^{n+1})$.

Then let $\mu_n=\sigma(\gamma_n)\mu_{n-1}\gamma_n^{-1}$ and
$\alpha_{n}=\alpha_{n-1}\cdot\gamma_n^{-1}$. We can easily see that
they satisfy the induction assumptions. Thus we can get a sequence
$\{\gamma_n\in U^n(k)|n\geq1\}$. As $\Lambda$ is complete,
$\prod_{n}\gamma_n$ converges to $\beta\in\Lambda$. It is not hard
to see that $\alpha\cdot\beta^{-1}=1$, and thus $\alpha\in\Lambda$.
\end{proof}

\begin{prop}
Let $R$ be a discrete valuation ring of characteristic $p$ with
fraction field $K$ and residue field $k$. Let $\mathcal{E}$ be an
$F$-crystal over $\text{Spec }R$. Let $\eta$ and $s$ be the generic
and closed point of $\text{Spec }R$. Assume that the first break
point of $NP(\mathcal{E})_\eta$ is $(1,m)$. Then the following two
conditions are equivalent:

(a) the Galois representation associated to $\mathcal{E}$ is
unramified, i.e., it factors through
$\phi:\text{Gal}_K\rightarrow\pi_1(\text{Spec }R)$.

(b) the first break point of $NP(\mathcal{E})_s$ is $(1,m)$.
\label{Trivialrepr}
\end{prop}
\begin{proof}
$Proof$: First consider $\text{Spec
}\overline{k}[[t]]\rightarrow\text{Spec }R$. By
Corollary~\ref{reducekernel} and ~\ref{compatiblerepr}, Condition
$(a)$ is equivalent to the triviality of the associated
representation of $\mathcal{E}|_{\text{Spec }\overline{k}[[t]]}$.
Moreover, as the Newton polygons of $E$ are preserved after pulled
back to $\text{Spec }\overline{k}[[t]]$,  Condition $(b)$ holds if
and only if the first break point of $NP(\mathcal{E}|_{\text{Spec
}\overline{k}[[t]]})_s$ is $(1,m)$. Hence it suffices to prove the
proposition for $R=k[[t]]$ with $k$ algebraically closed. Note that
in this case (a) is equivalent to the following (a)' the Galois
representation associated to $\mathcal{E}$ is trivial.

Condition (b)$\Rightarrow$ (a)':by \cite[Corollary 2.6.2]{Katz},
$\mathcal{E}$ is isogenous to an $F$-crystal $\mathcal{E}'$ which is
divisible by $p^m$, which contains a subcrystal $\mathcal{E}_1'$ of
rank $1$ and slope $m$. By Lemma~\ref{isogenous}, the Galois
representation in question is the same as the one associated to
$\mathcal{E}_1'|_{\text{Spec }K}$. By \cite[Theorem 2.7.4]{Katz},
$\mathcal{E}_1'$ becomes isogenous to a constant $F$-crystal over
$k((t))^{pf}$, and therefore the associated representation is
trivial.

(a)'$\Rightarrow$(b): By Lemma ~\ref{trivialcrystal},
$\mathcal{E}_{\text{Spec }K}$ has a trivial subcrystal of rank $1$
and slope $m$. Then we get an injection $\Phi:
\mathcal{L}_{\text{Spec} K}\rightarrow {\mathcal{E}}_{\text{Spec
}K}$, where $\mathcal{L}$ is a trivial $F$-crystal of rank $1$ and
slope $m$ over $\text{Spec }R$. Apply \cite[Main Theorem]{de Jong}
to $\mathcal{E}, \mathcal{L}$ and $\Phi$. We obtain a nontrivial map
$\mathcal{L}\rightarrow\mathcal{E}$. Restricting to $s$,we see that
$\mathcal{E}_s$ contains a subcrystal of rank $1$ and slope $m$. On
the other hand, by Grothendieck's specialization theorem
\cite[2.3.1]{Katz}, $NP(\mathcal{E})_s$ lies on or above
$NP(\mathcal{E})_\eta$. Hence $(1,m)$ is the first break point of
$NP(\mathcal{E})_s$.
\end{proof}

\section{The proof}
Assume the common break point is $P=(i,m)$. If we assume that
codim$(U, \{s\}^-)>1$, then we just need to show that $P$ is also a
break point of $NP(\mathcal{E})_s$.

Step $1$: Reduce to the special case when the common break point $P$
is of the form $(1,m)$.

In the general case, let $\mathcal{E}'=\wedge^i\mathcal{E}$. By
assumption, $(1,m)$ is the first break point of $NP(\mathcal{E}')_x$
for all $x\in U\backslash\{s\}$. Applying the result for the special
case, we obtain that $(1,m)$ is a break point of
$NP(\mathcal{E}')_s$, and hence $P$ is a break point of
$NP(\mathcal{E})_s$.

Step $2$: First as $S$ is locally noetherian, we may shrink $S$ to
an open affine neighborhood $\text{Spec }A$ of $s$ such that
$(\text{Spec }A\backslash\{s\})\subset U$. Then we follow the same
reduction steps as in the proof of \cite[Theorem 4.1]{dejong-oort}.
We obtain that there exists a Noetherian  complete local normal
domain $A$ of dimension 2 with algebraically closed residue field
$k$ and a morphism $\phi:\text{Spec }A\rightarrow S$  that maps
closed point to $s$ and other points into $U$. Hence it suffices to
prove the statement when $S$ is the spectrum of a Noetherian
complete local normal domain $A$ of dimension 2 with algebraically
closed residue field $k$, $s$ is the closed point and
$U=S\backslash\{s\}$.

Up to now, we have shown that it suffices to prove the following
simplified statement: \emph{let $A$ be a  Noetherian complete local
normal integral domain of dimension 2 with algebraically closed
residue field $k$. Let $s\in S=\text{Spec }A$ be the closed point
and $U=S\backslash\{s\}$. If $(1,m)$ is the first break point of
$NP(\mathcal{E})_x$ for every $x\in U$, then $(1,m)$ is the first
break point of $NP(\mathcal{E})_s$.}

Let $K$ be the fraction field of $A$. Consider the Galois
representation  $ \rho: \text{Gal}(\overline{K}/K)\rightarrow
\mathds{Z}_p^*$ defined in ~\ref{formaldefinition}. Let H be the
kernel of the composition of $\rho$ and
$\xymatrix{\mathds{Z}_p^*\ar[r]^{\text{mod }p}&\mathds{F}^*_p}$.
 Let $L$ be the subfield of $\overline{K}$ fixed
by $H$. As $\text{Gal}(\overline{K}/K)/H$ is a finite set, $L$ is a
finite Galois extension of $K$. Let $\widetilde{A}$ be the integral
closure of $A$ in $L$. By a standard argument, we see that
$\widetilde{A}$ is a Noetherian complete local normal domain of
dimension 2 with residue field $k$. Consider the finite morphism
$\text{Spec }\widetilde{A}\rightarrow\text{Spec }A$. It is not hard
to see that we only need to prove the statement for $\widetilde{A}$.
Replacing $A$ by $\widetilde{A}$, we may assume that
$\text{Im}(\rho)\subset 1+p\mathbb{Z}_p$ and that the homomorphism
$\xymatrix{\text{Gal}(\overline{K}/K)\ar[r]^\rho&\mathds{Z}_p^*
\ar[r]^{\text{log}}&\mathds{Z}_p}$ is valid.

Let $x\in U$. Assume that $\phi,\ \phi_x$ and $\psi_x$ are the same
as in ~\ref{Kernel}. Since $\mathcal{E}|_{\text{Spec
}\widehat{\mathcal{O}}_{U,x}}$ satisfies the assumptions and
Condition (b) in Proposition~\ref{Trivialrepr}, the associated
representation of $\mathcal{E}|_{\text{Spec
}\widehat{\mathcal{O}}_{U,x}}$, which is the composition of $\psi_x$
and $\rho$,  factors through $\phi_x$. It follows that
$\psi_x(\text{Ker}\phi_x)\subset\text{Ker}\rho$. By
Proposition~\ref{Trivialrepr}, we obtain that $\rho$ factors through
$\phi$. Thus we obtain a map
$\iota:\pi_1(U,\overline{\eta})\rightarrow\mathds{Z}_p^*\rightarrow\mathds{Z}_p$.

Take a resolution of singularities $\widetilde{S}\rightarrow S$; if
$A$ happens to be regular, let $\widetilde{S}$ be the blowup of the
special point of Spec $A$. Then the main result of \cite[Section
3]{dejong-oort} implies that $\iota$ can be extended to
$\widetilde{\iota}:\pi_1(\widetilde{S},\overline{\eta})\rightarrow
\mathds{Z}_p$. Let $\xi$ denote the generic point of a component of
the exceptional fibers of $\widetilde{S}\rightarrow S$.  Now we have
the following diagrams:$$\xymatrix{\text{Spec }K\ar[d]\ar[r]
&U\ar[d]&
&\text{Gal}(\overline{K}/K)\ar[d]\ar[r]\ar@{.>}[drr]&\pi_1(U,\overline{\eta})\ar[d]& \\
\text{Spec }\mathcal{O}_{\widetilde{S},\xi}\ar[r]
&\widetilde{S}&&\pi_1(\text{Spec
}\mathcal{O}_{\widetilde{S},\xi},\overline{\eta})\ar[r] &
\pi_1(\widetilde{S},\overline{\eta})\ar[r]&\mathds{Z}_p^*}$$
\begin{proof}
By definition the representation associated to
$\mathcal{E}_{\text{Spec }\mathcal{O}_{\widetilde{S},\xi}}$ is the
dotted arrow, and it is unramified  by the above commutative
diagram. By Proposition ~\ref{Trivialrepr} again, $(1,m)$ is the
first break point of $NP(\widetilde{S},\mathcal{E})_\xi$. Since
$\xi$ is mapped to $s$, $(1,m)$ is thus the first break point of
$NP(\mathcal{E})_s$.
\end{proof}

\Addresses

\begin{thebibliography}{99}

\bibitem{dejong-oort}
A.J. de Jong and F. Oort, \textit{Purity of the stratification by
Newton polygons}, Journal of AMS \textbf{13} (1999), no.1, pp.
209-241

\bibitem{de Jong} A.J. de Jong, \textit{Homomorphisms of Barsotti-Tate groups and crystals in positive
characteristic}, Invent. Math. \textbf{134} (1998), pp. 301-333

\bibitem{Katz}  N. Katz, \textit{Slope filtration of $F$-crystals},
in: Journ\'{e}es de G\'{e}om\'{e}trie Alg\'{e}briques (Rennes 1978),
Ast\'{e}risque \textbf{63} (1979), pp. 113-164

\bibitem{Berthelot}P.
Berthelot, \textit{Cohomologie Cristalline des Sch\'{e}mas de
Caract\'{e}ristique $p>0$}, Lec. Notes Math. \textbf{407},
Springer-Verlag(1974)

\bibitem{Berthelot-Messing} P.Berthelot and W.Messing, \textit{Th\'{e}orie de Dieudonn\'{e}
cristalline $\Rmnum{3}$}, in: The Grothendieck Festschrift
$\Rmnum{1}$, Progr. in Math. \textbf{86}, Birkh\"{a}user (1990),
pp.171-247

\bibitem{Vasiu} A.Vasiu, \textit{Crystalline boundedness principle},
Ann. Sci. \'Ecole Norm. Sup.(4) \textbf{39} (2006), no. 2, 245-300.

\bibitem{Serre} J.P. Serre, \textit{Local Fields}, GTM \textbf{67},
Springer

\bibitem{Matsumura} H. Matsumura, \textit{Commutative Algebra--Second
Edition}, Benjamin/Cummings Publishing Company, 1980

\bibitem{Milne} J.S. Milne, \textit{Etale Cohomology}, Princeton
University Press

\bibitem{Manin} Yu.I. Manin, \textit{The theory of commutative formal groups over fields
of finite characteristic}, Russian Math. Surveys \textbf{18},
(1963), 1-83

\bibitem{Rivano} N. Saavedra Rivano, \textit{Cat\'{e}gories
tannakiennes}, Lect. Notes Math. \textbf{265}, Springer-Verlag
(1972)

\bibitem{Grothendieck} A. Grothendieck, \textit{Rev\^{e}tements Etales et Groupe
Fondamental}, SGA \textbf{1}, Springer-Verlag

\bibitem{Neukirch} J\"{u}rgen Neukirch, \textit{Algebraic Number
Theory}, translated from the German by Norbert Schappacher,
Springer-Verlag 1999

\bibitem{Eisenbud} D. Eisenbud, \textit{Commutative Algebra with a view toward algebraic
geometry}, GTM \textbf{150}, Springer-Verlag

\bibitem{Oort} F. Oort, \textit{Purity reconsidered}, see: http://www.math.uu.nl/people/oort/

\bibitem{Zink} T. Zink, \textit{De Jong-Oort Purity for p-divisible
Groups}, to appear in the Manin Festschrift.
\end{thebibliography}
\end{document}